\documentclass[a4paper,11pt]{article}
\usepackage[T1]{fontenc}
\usepackage{lmodern,amsmath,amsthm,amsfonts,amssymb,graphicx,float,microtype,thmtools,underscore,mathtools,xurl,multirow}
\usepackage[usenames,dvipsnames,svgnames,table]{xcolor}
\usepackage[shortlabels]{enumitem}
\setlist[itemize]{topsep=0ex,itemsep=0ex,parsep=0ex}
\setlist[enumerate]{topsep=0ex,itemsep=0ex,parsep=0ex}
\usepackage[unicode=true]{hyperref}
\hypersetup{ 
colorlinks,
breaklinks=true,
linkcolor={blue!60!black},
citecolor={black},
urlcolor={blue!60!black},
pdftitle={Defective and Clustered Colouring of Graphs with Given Girth}}
\usepackage[capitalise, compress, nameinlink, noabbrev]{cleveref}
\crefname{lem}{Lemma}{Lemmas}
\crefname{thm}{Theorem}{Theorems}
\crefname{prop}{Proposition}{Propositions}
\crefname{cor}{Corollary}{Corollaries}
\newcommand{\defn}[1]{\textcolor{Maroon}{\emph{#1}}}

\newcommand{\bigchi}{\raisebox{1.55pt}{\scalebox{1.2}{\ensuremath\chi}}}
\newcommand{\cchi}{\bigchi_{\star}\hspace*{-0.2ex}}

\newcommand{\dchi}{\bigchi\hspace*{-0.1ex}_{\Delta}\hspace*{-0.2ex}}

\usepackage[longnamesfirst,numbers,sort&compress]{natbib}
\makeatletter
\def\NAT@spacechar{~}
\makeatother
\usepackage[margin=30mm]{geometry}
\renewcommand{\baselinestretch}{1.09}
\allowdisplaybreaks
\DeclarePairedDelimiter{\floor}{\lfloor}{\rfloor}
\DeclarePairedDelimiter{\ceil}{\lceil}{\rceil}

\renewcommand{\epsilon}{\varepsilon}
\renewcommand{\emptyset}{\varnothing}

\renewcommand{\geq}{\geqslant}
\renewcommand{\leq}{\leqslant}

\DeclareMathOperator{\dist}{dist}
\DeclareMathOperator{\col}{col}
\DeclareMathOperator{\tw}{tw}

\DeclareMathOperator{\fvn}{fvn}
\DeclareMathOperator{\pw}{pw}
\DeclareMathOperator{\td}{td}
\DeclareMathOperator{\cir}{circ}
\DeclareMathOperator{\had}{had}

\newcommand{\GG}{\mathcal{G}}

\newcommand{\NN}{\mathbb{N}}

\newcommand{\TT}{\mathcal{T}}
 
\renewcommand{\thefootnote}{\fnsymbol{footnote}}
\theoremstyle{plain}
\newtheorem{thm}{Theorem}
\newtheorem{lem}[thm]{Lemma}
\newtheorem{cor}[thm]{Corollary}
\newtheorem{prop}[thm]{Proposition}

\crefname{obs}{Observation}{Observations}
\newtheorem*{lem*}{Lemma}
\theoremstyle{definition}

\newtheorem*{conj*}{Conjecture}

\begin{document}
\title{\bf\boldmath\fontsize{18pt}{18pt}\selectfont Defective and Clustered Colouring\\ of Graphs with Given Girth}

\author{%
Marcin~Bria\'nski\,\footnotemark[3] \qquad
Robert~Hickingbotham\,\footnotemark[2] \qquad 
David~R.~Wood\,\footnotemark[1]
}

\maketitle

\begin{abstract}
The \defn{defective chromatic number} of a graph class $\GG$ is the minimum integer $k$ such that for some integer $d$, every graph in $\GG$ is $k$-colourable such that each monochromatic component has maximum degree at most $d$. Similarly, the \defn{clustered  chromatic number} of a graph class $\GG$ is the minimum integer $k$ such that for some integer $c$, every graph in $\GG$ is $k$-colourable such that each monochromatic component has at most $c$ vertices. This paper determines or establishes bounds on the defective and clustered chromatic numbers of  graphs with given girth in minor-closed classes defined by the following parameters: Hadwiger number, treewidth, pathwidth, treedepth, circumference, and feedback vertex number. One striking result is that for any integer $k$, for the class of triangle-free graphs with treewidth $k$, the  defective chromatic number, clustered chromatic number and chromatic number are all equal. The same result holds for graphs with treedepth $k$, and generalises for graphs with no $K_p$ subgraph. 
\end{abstract}

\footnotetext[3]{Department of Theoretical Computer Science, Jagiellonian University, Kraków, Poland (\texttt{marcin.brianski@doctoral.uj.edu.pl}). M.~Bria\'nski was partially supported by a Polish National Science Center grant (BEETHOVEN; UMO-2018/31/G/ST1/03718).}


\footnotetext[2]{D\'epartement d'Informatique, Universit\'e libre de Bruxelles, Belgium ({\tt robert.hickingbotham@ulb.be}). This research was completed when R.\ Hickingbotham was supported by Australian Government Research Training Program Scholarship.}

\footnotetext[1]{School of Mathematics, Monash University, Melbourne, Australia (\texttt{david.wood@monash.edu}). Research of Wood supported by the Australian Research Council and by NSERC.}

 \newpage
 
\section{Introduction}
\label{Introduction}
\renewcommand{\thefootnote}{\arabic{footnote}}

This paper studies improper colourings of graphs in minor-closed classes with given  girth. First, we give the essential definitions. A \defn{colouring} of a graph\footnote{We consider simple, finite, undirected graphs~$G$ with vertex-set~${V(G)}$ and edge-set~${E(G)}$. See latter sections and \citep{Diestel5} for graph-theoretic definitions not given here.}  $G$ is a function that assigns one colour to each vertex of $G$. For an integer $k\geq 1$, a \defn{$k$-colouring} is a colouring using at most $k$ colours. A colouring of a graph is \defn{proper} if each pair of adjacent vertices receives distinct colours. The \defn{chromatic number $\bigchi(G)$} of a graph $G$ is the minimum integer $k$ such that $G$ has a proper $k$-colouring. For a graph class $\GG$, the \defn{chromatic number $\bigchi(\GG)$} is $\max\{\bigchi(G):G\in\GG\}$, and is \defn{unbounded} if the maximum does not exist. A \defn{monochromatic component} with respect to a colouring of a graph $G$ is a connected component of the subgraph of $G$ induced by all the vertices assigned a single colour. A colouring has \defn{defect} $d$ if every monochromatic component has maximum degree at most $d$. A colouring has \defn{clustering} $c$ if every monochromatic component has at most $c$ vertices. Note that a colouring with defect 0 or clustering 1 is precisely a proper colouring. The \defn{defective chromatic number $\dchi(\GG)$} of a graph class $\mathcal{G}$ is the minimum integer $k$ for which there exists an integer $d$ such that every graph in  $\GG$ is $k$-colourable with defect $d$, and is \defn{unbounded} if no such $k$ exists. The \defn{clustered chromatic number $\cchi(\GG)$} of a graph class $\mathcal{G}$ is the minimum integer $k$ for which there exists an integer $c$ such that every graph in  $\GG$ is $k$-colourable with clustering $c$, and is \defn{unbounded} if no such $k$ exists. 
By definition, $\dchi(\GG)\leq\cchi(\GG)\leq\bigchi(\GG)$. See \citep{WoodSurvey} for an extensive survey on this topic.

We consider defective and clustered colouring of graphs with given girth, where our aim is to minimise the number of colours rather than attain specific defect or clustering bounds\footnote{Other papers that focus on specific defect or clustering bounds for graphs of given girth include \citep{CE19,Skrekovski00,Skrekovski99a,Skrekovski99,CCJS17}.}. First note that for any integer $g$, the class of graphs with girth at least $g$ has unbounded defective chromatic number. To see this, suppose that for some $g$ and $k$, the class of graphs with girth at least $g$ has defective chromatic number at most $k$. That is, for some $d$, every graph $G$ with girth at least $g$ is $k$-colourable with defect at most $d$. Each monochromatic component is properly $(d+1)$-colourable, so $G$ is properly $k(d+1)$-colourable. This contradicts the result of \citet{Erdos59}, which states that graphs with girth $g$ have unbounded chromatic number. Hence the class of graphs with girth at least $g$ has unbounded defective chromatic number, and thus has unbounded clustered chromatic number. This negative result motivates studying defective and clustered colourings of graphs with given girth and satisfying some other structural assumptions. To obtain interesting results, it is not enough to just consider degeneracy, since  \citet[Property~6]{KosNes99} proved that for all $k,d,g$ there is a $k$-degenerate graph with girth at least $g$ that is not $k$-colourable with defect $d$. Thus chromatic number, defective chromatic number and clustered chromatic number of the class of $k$-degenerate graphs with girth $g$ equals $k+1$.

Here we focus on minor-closed graph classes. Defective and clustered colourings of graphs excluding a fixed minor are widely studied~\citep{DEMWW22,DEMW23,OOW19,Liu24,KM07,Norin15,vdHW18,LW1,LW2,LW3,LW4,Kawa08,Wood10,EKKOS15,LO18,KO19,DN17,NSSW19,NSW22,CE19}. We focus on the following minor-monotone parameters: 
treewidth $\tw(G)$, 
pathwidth $\pw(G)$, 
treedepth $\td(G)$, 
feedback vertex number $\fvn(G)$, 
circumference $\cir(G)$, 
and Hadwiger number $\had(G)$. 
See \cref{Definitions} for these well-known definitions. For now, note the following inequalities:
\begin{align}
    \had(G)-1 \leq \tw(G) & \leq \pw(G) \leq \td(G)-1     \label{Parameters}\\
    \had(G)-1 \leq \tw(G) & \leq \fvn(G)+1\\
    \had(G)-1 \leq \tw(G) & \leq \cir(G)-1.
\end{align}

\subsection{Defective Colouring of Bounded Treewidth Graphs}

Here we summarise results on defective colouring of minor-closed graph classes; refer to \cref{TableTreewidthDefective,TableMinorDefective}. Original results in this paper are indicated by $\ast$.

\begin{table}[!ht]
\caption{Defective chromatic number of bounded treewidth classes with respect to girth (ignoring lower order additive terms).}
\label{TableTreewidthDefective}
\centerline{
\begin{tabular}{c|ccccc}
\hline
girth & fvn $k$ & circumference $k$ & treedepth $k$ & pathwidth $k$ & treewidth $k$ \\
\hline
3 & $3^{\ast}$ & $\Theta(\log k)$ & $k$ & $k+1$ & $k+1$ \\
4 & $3^{\ast}$ & $\Omega(\sqrt{\log k})^{\ast}\dots O(\log k)$ & $\sqrt{2k}^{\ast}\dots\ceil{\frac{k+2}{2}}^{\ast}$ & $\sqrt{2k}^{\ast}\dots\ceil{\frac{k+3}{2}}^{\ast}$ &
$\ceil{\frac{k+3}{2}}^{\ast}$ \\
$\geq 5$ & 2 & 2 & 2 & 2 & 2 \\
\hline
\end{tabular}}
\end{table}

First consider the case with no girth constraint (girth 3). Graphs with treewidth $k$ are $k$-degenerate and thus properly $(k+1)$-colourable. This implies that  graphs with pathwidth $k$ are $(k+1)$-colourable, and graphs with treedepth $k$ are properly $k$-colourable. Standard examples (see \Cref{sect:stdexamples} and \citep{WoodSurvey}) show that these bounds cannot be improved for defective or clustered colouring. \citet{MRW17} showed that graphs with circumference $k$ are $O(\log k)$-colourable with clustering $k$, and thus with defect $k-1$. The number of colours is tight up to the leading constant. Thus the defective and clustered chromatic number of the class of graphs with circumference $k$ is in $\Theta(\log k)$. Every graph $G$ with $\fvn(G)\leq k$ is 3-colourable with clustering $k$ and defect $k-1$: just properly 2-colour the underlying forest and give a third colour to each apex vertex. \cref{FVN1TriangleFreeNotDefectLowerBound} shows that three colours is best possible even for feedback vertex number 1. Thus for each integer $k\geq 1$, the clustered and defective chromatic number of the class of  graphs $G$ with $\fvn(G)\leq k$ equals 3.

Now consider the triangle-free case (girth 4). \citet{DK17} proved the following important result. 

\begin{thm}[\citep{DK17}]
\label{DK}
For any $k\in\NN$ the chromatic number of the class of triangle-free graphs with treewidth $k$ equals $\ceil{\frac{k+3}{2}}$.
\end{thm}

We extend their lower bound to prove that for the class of triangle-free graphs with treewidth $k$, the defective chromatic number, the clustered chromatic number, and the chromatic number all equal $\ceil{\frac{k+3}{2}}$ (see \cref{TriangleFreeTreewidthDefectiveClusteredProper} which in fact holds for any fixed excluded complete subgraph). We prove the following similar result for treedepth: for the class of triangle-free graphs with treedepth $k$, the defective chromatic number, the clustered chromatic number, and the proper chromatic number are all equal, and this value is between $(1-o(1))\sqrt{2k}$ and $\ceil{\frac{k+2}{2}}$ (see \cref{TriangleFreeTreedepth}). The lower bound for triangle-free graphs with treedepth $k$ also leads to a $\Omega(\sqrt{\log k})$ lower bound for triangle-free graphs with circumference $k$ (see \cref{CircumferenceLowerBound}). So the defective and clustered chromatic number of the class of triangle-free graphs with circumference $k$ is between $\Omega(\sqrt{\log k})$ and  $O(\log k)$. Finally, \cref{FVN1TriangleFreeNotDefectLowerBound} shows that three colours is best possible even for triangle-free graphs with feedback vertex number 1. Thus for each integer $k\geq 1$, the clustered and defective chromatic number of the class of triangle-free graphs $G$ with $\fvn(G)\leq k$ equals 3. 

Now consider the girth at least 5 case. The following result of \citet{OOW19} is paramount. Here $\nabla(G)$ is the maximum average degree of a graph $H$ such that the 1-subdivision of $H$ is isomorphic to a subgraph of $G$. 

\begin{thm}[\citep{OOW19}]
\label{OOW}
Every $K_{s,t}$-subgraph-free graph $G$ is $s$-colourable with defect at most some function $f(s,t,\nabla(G))$.
\end{thm}

We give a short proof of a result that is equivalent to \cref{OOW} in \cref{StrongColouringNumbers}. 

Every graph with girth at least $5$ contains no $K_{2,2}$ subgraph. Thus \cref{OOW} with $s=t=2$ implies that for any graph class $\GG$ with bounded $\nabla$ (which includes every minor-closed class), the  graphs in $\GG$ with girth at least 5 are $2$-colourable with bounded defect. Girth 5 here cannot be reduced to 4, even for graphs with bounded treedepth, since we prove that the class of triangle-free graphs with treedepth $k$ has defective chromatic number at least $\Omega(\sqrt{k})$ (see \cref{TriangleFreeTreedepthProperLowerBound}). 
By \eqref{Parameters}, the same result holds with `treedepth' replaced by `pathwidth', `treewidth' or `Hadwiger number'. For each of these parameters,  there is a qualitative change in behaviour from girth 4 to 5. For a class of graphs with girth at least 5 and with any of these parameters bounded, the defective chromatic number equals 2. On the other hand, for graphs with girth 4, the defective chromatic number increases with the parameter. 

\subsection{Defective Colouring of General Minor-Closed Classes}

Now consider defective colourings of planar graphs and more general minor-closed classes; the results are summarised in \cref{TableMinorDefective}.

\begin{table}[!ht]
\begin{center}
\caption{Defective chromatic number of minor-closed classes with respect to girth.}
\label{TableMinorDefective}
\begin{tabular}{c|cccc}
\hline
girth & planar graphs & Euler genus $\gamma$ & $(\gamma,k)$-apex & Hadwiger number $t$\\
\hline
3 & $3$ & $3$ & $4$ & $t$\\
4 & 2 or 3 & 2 or 3 & 2, 3 or 4 & $\ceil{\frac{t+2}{2}}^{\ast}\dots t$   \\
$\geq 5$ & $2$ & $2$ & 2 & 2\\
\hline
\end{tabular}
\end{center}
\end{table}

First consider the case with no girth constraint. \citet{CCW86} proved that planar graphs are 3-colourable with defect 2 (see \citep{Poh90,EH99,vdHW18} for various strengthenings). More generally, \citet{Archdeacon87} proved that graphs of Euler genus $\gamma$ are 3-colourable with defect $O(\gamma)$, and \citet{CGJ97} improved the defect bound to $O(\sqrt{\gamma})$, which is best possible. Standard examples provide, for any integer $d\geq 0$, a planar graph that is not 2-colourable with defect $d$ (see \citep{WoodSurvey}). Thus for each integer $\gamma\geq 0$, the defective chromatic number of the class of graphs with Euler genus $\gamma$ equals $3$. 

Now consider apex graphs. For integers $\gamma,k\geq 0$, a graph $G$ is \defn{$(\gamma,k)$-apex} if there exists $A\subseteq V(G)$ such that $|A|\leq k$ and $G-A$ has Euler genus at most $\gamma$. A $(0,k)$-apex graph is called \defn{$k$-apex}, in which case $G-A$ is planar. Every $(\gamma,k)$-apex graph is 4-colourable with defect $O(\max\{\sqrt{\gamma},k\})$: just apply the above 3-colour result for graphs of Euler genus $\gamma$, and use a fourth colour for the $k$ apex vertices. Standard examples show that four colours is best possible, even for $(0,1)$-apex graphs (see \citep{WoodSurvey}). 

Now consider the class of graphs $G$ with Hadwiger number $\had(G)\leq t$. \citet{Hadwiger43} famously conjectured that such graphs are properly $t$-colourable. The best known upper bound on the chromatic number is $O(t\log \log t)$ due to \citet{DP25}. Standard examples provide a lower bound of $t$ on the defective (and thus clustered) chromatic number of this class. \citet{EKKOS15} proved that graphs $G$ with $\had(G)\leq t$ are $t$-colourable with defect $O(t^2\log t)$. Thus the defective chromatic number of the class of graphs $G$ with $\had(G)\leq t$ equals $t$. The defect bound in the above result of \citet{EKKOS15} was improved to $O(t)$ by \citet{vdHW18}.

Finally, consider the triangle-free (girth 4) case. Each of the upper bounds in the triangle-free case of 
\cref{TableMinorDefective} follow from the analogous upper bounds without the triangle-free assumption. The lower bounds of 2 are trivial. The lower bound of $\ceil{\frac{t+2}{2}}$ for graphs with Hadwiger number $t$ follows from the fact that the defective chromatic number of the class of triangle-free graphs with treewidth $k$ equals $\ceil{\frac{k+3}{2}}$ and since $\had(G)\leq\tw(G)+1$ for every graph $G$.

\subsection{Clustered Colouring of Bounded Treewidth Graphs}

Here we consider clustered colourings of various bounded treewidth graph classes with respect to girth. \cref{TableTW} summarises the results. 

\begin{table}[!ht]
\setlength{\tabcolsep}{1.5mm}
\caption{Clustered chromatic number of bounded treewidth classes with respect to girth (ignoring lower order additive terms).}
\label{TableTW}
\centerline{
\begin{tabular}{c|ccccc}
\hline
girth & fvn $k$ & circumference $k$ & treedepth $k$ & pathwidth $k$ & treewidth $k$ \\
\hline
3 & 3 & $\Theta(\log k)$ & $k$ & $k+1$ & $k+1$ \\
4 & 
3  & 
$\Omega(\sqrt{\log k})^{\ast} \dots O(\log k)$  & $\sqrt{2k}^{\ast}\dots\ceil{\frac{k+2}{2}}^{\ast}$ &
$\sqrt{2k}^{\ast}\dots\ceil{\frac{k+3}{2}}^{\ast}$ &
$\ceil{\frac{k+3}{2}}^{\ast}$ \\
5 or 6 & $3^{\ast}$ & $2^{\ast}$ & $2^{\ast}$ & 3 & 3 \\
$\geq 7$ & $2^{\ast}$ & $2^{\ast}$ & $2^{\ast}$ & 2 or 3 & 2 or 3 \\
$\Omega(\log k)$ & $2^{\ast}$& $2^{\ast}$& $2^{\ast}$& $2^{\ast}$ & $2^{\ast}$ \\
\hline
\end{tabular}}
\end{table}

\citet{LW2} proved the following analogue of \cref{OOW} for clustered colouring of bounded treewidth graphs. 

\begin{thm}[\citep{LW2}]
\label{LiuWoodHplanar}
For any $s, t, k \in \NN$, every $K_{s,t}$-subgraph-free graph of treewidth at most $k$ is  $(s+1)$-colourable with clustering at most some function $f(s,t,k)$.
\end{thm}

As before, \cref{LiuWoodHplanar} with $s=t=2$ implies every graph with girth at least 5 and with treewidth at most $k$ is 3-colourable with clustering at most some function $f(k)$. In both \cref{LiuWoodHplanar,LiuWoodH} the number of colours is best possible, but the extremal examples have many triangles. 

Consider the triangle-free case (girth 4). As mentioned above, we extend \cref{DK} to show that the clustered chromatic number of triangle-free graphs with treewidth $k$ equals $\ceil{\frac{k+3}{2}}$. This implies the same upper bound for triangle-free graphs with pathwidth $k$, and an upper bound of $\ceil{\frac{k+2}{2}}$ for triangle-free graphs with treedepth $k$. As mentioned above, \citet{MRW17} showed that graphs with circumference $k$ are $O(\log k)$-colourable with clustering $k$. We show that this bound on the number of colours is best possible even for triangle-free graphs (see  \cref{CircumferenceLowerBound}). Thus the clustered chromatic number of the class of triangle-free graphs with circumference $k$ is in $\Theta(\log k)$. For triangle-free graphs with $\fvn(G)\leq k$ the upper bound of 3 is best possible, even for $k=1$ (see \cref{FVN1TriangleFreeNotDefectLowerBound}). 

Now consider the case of girth at least 5. \cref{OOW} implies that graphs with girth at least 5 and with bounded treedepth are 2-colourable with bounded clustering, since in a graph of bounded treedepth, every path has bounded length (see \citep{Sparsity}). This observation in fact implies that for every graph class $\GG$ with bounded treedepth, $\chi_*(\GG)=\chi_\Delta(\GG)$. Similarly, we improve the number of colours from 3 to 2 in the case of graphs with girth at least 5 and circumference $k$  (see \cref{CircumferenceGirth5}), and in the case of graphs with girth at least 7 and feedback vertex number $k$  (see \cref{FVNgirth7}).  The upper bound of 2 for graphs with treewidth~$k$ and girth $\Omega(\log k)$  in \cref{TableTW} follows from our more general results for graphs with Hadwiger number $t$ and with girth $\Omega(\log t)$ (see \cref{BigGirth}).


\subsection{Clustered Colouring of General Minor-Closed Classes}

Here we consider clustered colourings of various minor-closed graph classes (with unbounded treewidth) with respect to girth. \cref{TableMinor} summarises the results. 

\begin{table}[!ht]
\begin{center}
\caption{Clustered chromatic number of minor-closed classes with respect to girth.}
\label{TableMinor}
\begin{tabular}{c|cccc}
\hline
girth & planar graphs & Euler genus $\gamma$ & $(\gamma,k)$-apex & Hadwiger number $t$\\
\hline
3 & $4$ & $4$ & $5$ & $t$\\
4 & $3$ & $3$ & 3 or 4 & $\ceil{\frac{t+2}{2}}^{\ast}\dots t$   \\
5 or 6 & $2$ & $2$ & $3^{\ast}$ & 3 or 4\\
7 or 8 & $2$ & $2$ & $2^{\ast}$ or $3^{\ast}$ & 2, 3 or 4\\
$\geq 9$ & $2$ & $2$ & $2^{\ast}$ & 2, 3 or 4\\
$\Omega(\log t)$ &&&& $2^{\ast}$\\
\hline
\end{tabular}
\end{center}
\end{table}

First consider planar graphs. The  4-Colour Theorem says that planar graphs are properly 4-colourable~\citep{RSST97}. A much simpler proof by \citet{CCW86} shows that  planar graphs are 4-colourable with clustering 2. Standard examples show that the clustered chromatic number of the class of planar graphs equals 4 (see \citep{WoodSurvey}). Similarly, Grotzsch's Theorem~\citep{Grotzsch7,Thomassen03,Thomassen94,SY89} says that triangle-free planar graphs are properly 3-colourable. As a lower bound, for any integer $d\geq 0$,  \citet{Skrekovski99} constructed triangle-free planar graphs that are not 2-colourable with defect $d$; thus the defective and clustered chromatic number of triangle-free planar graphs equals 3. (Note that \cref{SSS} below implies that triangle-free planar graphs are 3-colourable with bounded clustering, without using Grotzsch's Theorem). 

Now consider girth $g\geq 5$. \citet{LW2} proved the following analogue of \cref{LiuWoodHplanar}.

\begin{thm}[\citep{LW2}]
\label{LiuWoodH}
For any $s,t\in\NN$ and for any graph $X$, every $K_{s,t}$-subgraph-free $X$-minor-free graph is $(s+2)$-colourable with clustering at most some function $f(s,t,X)$.
\end{thm}

As before, \cref{LiuWoodH} with $s=t=2$ implies that every $X$-minor-free graph with girth at least 5 is 4-colourable with clustering at most some function $f(X)$.



Now consider graphs embeddable on surfaces. 
 \citet{EO16} and \citet{KT12} proved that graphs of bounded Euler genus are 5-colourable with bounded clustering. Improving these results, \citet{DN17} showed that every graph with Euler genus $\gamma$ is 4-colourable with clustering $O(\gamma)$. Hence the clustered chromatic number of the class of graphs embeddable on any fixed surface equals 4.  The proof is based on the following lemma, which is proved using the so-called island method (see \citep{WoodSurvey} for a concise presentation).

\begin{lem}[\citep{DN17}]
\label{SSS}
Let $\GG$ be a hereditary graph class such that for some real numbers $c,\epsilon,\delta>0$ and some integers $k,n_0\geq 2$, for all $n\geq n_0$, every $n$-vertex graph in $\GG$ has a balanced separator of order at most $cn^{1-\epsilon}$, and has at most $(k-\delta)n$ edges. Then every graph in $\GG$ is $k$-colourable with clustering at most some function $f(k,c,\epsilon,\delta,n_0)$.
\end{lem}

The balanced separator assumption in \cref{SSS} is satisfied with $\epsilon=\frac12$ for planar graphs~\cite{LT79,LT80}, for graphs of bounded genus~\cite{GHT-JAlg84,Djidjev85a}, and for any minor-closed class~\cite{AST90}.

The above-mentioned clustered  4-colour theorem of \citet{DN17} is implied by \cref{SSS}   with $\delta=\frac12$ and 
$n_0:=6(\gamma-2)$, since then it follows from Euler's formula that every graph $G$ with $n\geq 3$ vertices and Euler genus $\gamma$ satisfies  $|E(G)|\leq 3(n+\gamma-2)\leq (4-\delta)n$ for $n\geq n_0$. Similarly, every triangle-free graph  with $n\geq 3$ vertices and Euler genus $\gamma$ has at most $2(n+\gamma-2)$ edges. Thus such graphs are 3-colourable with bounded clustering by \cref{SSS}. More generally, every $n$-vertex graph  with Euler genus $\gamma$ and girth $g$ has at most $\frac{g}{g-2}(n+\gamma-2)$ edges. In particular, if $g\geq 5$, then $G$ has at most $\frac53(n+\gamma-2)$ edges, implying such graphs are 2-colourable with bounded clustering by \cref{SSS}. These results were mentioned by \citet{WoodSurvey}. 

Our results are summarised in \cref{TableMinor}. Every $(\gamma,k)$-apex graph $G$ is 5-colourable with bounded clustering: 4-colour $G-A$ using \cref{SSS} and use one new colour for $A$. Standard examples show that the clustered chromatic number of $(0,1)$-apex graphs equals 5. Similarly, triangle-free $(\gamma,k)$-apex graphs are 4-colourable with bounded clustering. Our other proofs for $(\gamma,k)$-apex graphs are in \cref{ApexGraphs}.
For $(0,5)$-apex graphs with girth at least $6$,
\cref{FeedbackLowerBound} implies that the clustered chromatic number is at least $3$.



Now consider the clustered chromatic number of the class of graphs $G$ with $\had(G)\leq t$. As mentioned above, standard examples show a lower bound of $t$. \citet{KM07} first proved a $O(t)$ upper bound. Their bound was $\ceil{\frac{31}{2}(t+1)}$, which was improved to $\ceil{\frac{7t+4}{2}}$ by \citet{Wood10}\footnote{The result of \citet{Wood10} depended on a result announced by Norin and Thomas~\cite{Thomas09}, which has not yet been fully written.}, to $4t$ by \citet{EKKOS15}, to $3t$ by \citet{LO18}, to $2t$ independently by \citet{Norin15}, \citet{vdHW18} and \citet{DN17}, to $t+1$ by \citet{LW3}, and to $t$ by \citet{DEMW23} (announced earlier by \citet{DN17}).
Thus the clustered chromatic number of the class of graphs $G$ with $\had(G)\leq t$ equals $t$. 

Finally, we mention the final rows in \cref{TableTW,TableMinor}, where the girth is allowed to grow slowly with the excluded minor. In \cref{BigGirth}, we prove that every graph with Hadwiger number at most $t$ and with girth $\Omega(\log t)$ is 2-colourable with clustering 2 (see \cref{GirthLogt}). This section also explores results on proper colouring of graphs with given Hadwiger number and given girth. We show, via a result of \citet{KO03}, that for every integer $g\geq 5$ there exists $c_g\in(0,1)$ with $c_g\to 0$, such that graphs with Hadwiger number $t$ and girth $g$ are properly $O(t^{c_g})$ colourable, thus asymptotically improving on Hadwiger's Conjecture for girth at least 5.

\section{Definitions}
\label{Definitions}

We consider finite simple undirected graphs $G$ with vertex-set $V(G)$ and edge-set $E(G)$. For a vertex $v\in V(G)$, let $N_G(v):=\{w\in V(G): vw\in E(G)\}$ and $N_G[v]:=N_G(v)\cup\{v\}$. 

A \defn{graph class} $\GG$ is a set of graphs closed under isomorphism.  A graph class $\GG$ is \defn{proper} if some graph is not in $\GG$. A graph class $\GG$ is \defn{hereditary} if it is closed under induced subgraphs. 

A graph $G$ is \defn{$k$-degenerate} if every subgraph of $G$ has minimum degree at most $k$. Such graphs are properly $(k+1)$-colourable by a greedy algorithm.

A graph $H$ is a \defn{minor} of a graph $G$ if $H$ is isomorphic to a graph that can be obtained from a subgraph of $G$ by contracting edges. A graph~$G$ is \defn{$H$-minor-free} if~$H$ is not a minor of~$G$. The \defn{Hadwiger number} of a graph $G$, denoted \defn{$\had(G)$}, is the maximum integer $t$ such that $K_t$ is a minor of $G$. A graph class $\GG$ is \defn{minor-closed} if for every graph $G\in\GG$ every minor of $G$ is in $\GG$. Similarly, a graph parameter $f$ is \defn{minor-monotone} if $f(H)\leq f(G)$ for every graph $G$ and every minor $H$ of $G$.

A \defn{surface} is a compact 2-dimensional manifold. For any fixed surface $\Sigma$, the class of graphs embeddable on $\Sigma$ (without crossings) is minor-closed. A surface with $h$ handles and $c$ cross-caps has \defn{Euler genus} $2h+c$. The \defn{Euler genus} of a graph $G$ is the minimum Euler genus of a surface in which $G$ embeds. Euler genus is a minor-monotone graph parameter. See \cite{MoharThom} for more about graph embeddings in surfaces.

A \defn{tree-decomposition} of a graph $G$ is a collection $\TT=(B_x :x\in V(T))$ of subsets of $V(G)$ (called \defn{bags}) indexed by the vertices of a tree $T$, such that (a) for every edge $uv\in E(G)$, some bag $B_x$ contains both $u$ and $v$, and (b) for every vertex $v\in V(G)$, the set $\{x\in V(T):v\in B_x\}$ induces a non-empty (connected) subtree of $T$. The \defn{width} of $\TT$ is $\max\{|B_x| \colon x\in V(T)\}-1$. The \defn{treewidth} of a graph $G$, denoted \defn{$\tw(G)$}, is the minimum width of a tree-decomposition of $G$.  

A \defn{path-decomposition} is a tree-decomposition in which the underlying tree is a path, simply denoted by the corresponding sequence of bags $(B_1,\dots,B_n)$. The \defn{pathwidth} of a graph $G$, denoted \defn{$\pw(G)$}, is the minimum width of a path-decomposition of $G$. 
By definition, $$\tw(G)\leq\pw(G).$$ 

A tree $T$ with a distinguished vertex $r\in V(T)$ is said to be \defn{rooted} (at $r$). The \defn{depth} of a vertex $v$ in $T$ is the number of vertices in the $rv$-path in $T$. For distinct vertices $u,v$ in $T$, if $v$ is on the $ur$-path in $T$, then $v$ is an \emph{ancestor} of $u$ and $u$ is a \emph{descendant} of $v$ in $T$. For each vertex $v$ in a rooted tree $T$, the subtree induced by $v$ and all its descendents is the \defn{subtree of $T$ rooted at} $v$. A forest is \emph{rooted} if each component tree is rooted. The \defn{depth} of a rooted forest $F$ is the maximum depth of a vertex in $F$. The \defn{closure} of $F$ is the graph with vertex set $V(F)$, where $vw$ is an edge if and only if $v$ is an ancestor of $w$ or vice versa. The \defn{treedepth} of a graph $G$, denoted \defn{$\td(G)$}, is the minimum depth of a rooted forest $F$ such that $G$ is a spanning subgraph of the closure of $F$. It is well-known that for every graph $G$,
$$\pw(G)\leq \td(G)-1.$$ To see this, introduce one bag for each root--leaf path, ordered left-to-right according to a plane drawing of $F$, producing a path-decomposition of $G$. 
Finally, if $G$ is a subgraph of the closure of a rooted tree $T$, and $v$ is a vertex of $T$, then the \defn{subgraph of $G$ rooted at} $v$ is the subgraph of $G$ induced by the vertices of the subtree of $T$ rooted at $v$. 

For a graph $G$, a set $A\subseteq V(G)$ is a \defn{feedback vertex set} of $G$ if $G-A$ is a forest. The \defn{feedback vertex number} of a graph $G$, denoted \defn{$\fvn(G)$}, is the minimum number of vertices in a feedback vertex set of $G$. Since a forest has treewidth at most 1, $$\tw(G)\leq\fvn(G)+1.$$

The \defn{circumference $\cir(G)$} of a graph $G$ that is not a forest is the length of the longest cycle in $G$. The \defn{circumference} of a forest is 2. \citet{Birmele03} showed that $$\tw(G)\leq\cir(G)-1;$$ 
see \citep{MarshallWood15,BJMMSS} for related results.

It is well-known and easily seen that treewidth, pathwidth, treedepth, feedback vertex number and circumference are minor-monotone.




\section{Lower Bounds}
\label{LowerBounds}

\subsection{Generalised Standard Examples}
\label{sect:stdexamples}

Many authors have noted that closures of rooted trees provide lower bounds on the defective and clustered chromatic  numbers~\citep{WoodSurvey,HS06,EKKOS15,OOW19,NSSW19,vdHW18}. Such graphs have been dubbed `standard examples'. The following definition and theorem generalises these results.


Let $H$ be a graph equipped with a vertex ordering $V(H)=\{v_1,\dots,v_n\}$. For an integer $d\geq 1$, let $H^{(d)}$ be the following graph. Start with a complete $d$-ary tree $T$ of vertex-height $n$. Let $V(H^{(d)}):=V(T)$, where for each root-leaf path $(x_1,\dots,x_n)$ in $T$, 
$x_ix_j\in E(H^{(d)})$  if and only if $v_iv_j\in E(H)$. For example, $K_n^{(d)}$ is the closure of the complete $d$-ary tree of height $n$. Numerous papers have noted that $K_n^{(d+1)}$ is not $(n-1)$-colourable with defect $d$ (see \citep{WoodSurvey}). Here we prove the following generalisation.

\begin{thm}
\label{GeneralisedStandardExample}
Let $H$ be a graph equipped with a vertex ordering $V(H)=\{v_1,\dots,v_n\}$. Fix integers $d,k\geq 1$. Then the following are equivalent:
\begin{enumerate}[(1)]
\item $H^{(d)}$ is $k$-colourable with defect $d-1$, 
\item $H^{(d)}$ is $k$-colourable with clustering $d$, 
\item $H^{(d)}$ is properly $k$-colourable,
\item $H$ is properly $k$-colourable.
\end{enumerate}
\end{thm}

\begin{proof}
We first prove that (4) $\Longrightarrow$ (3). Consider any proper $k$-colouring of $H$. Let $T$ be the rooted tree underlying $H^{(d)}$. For each root--leaf path $(x_1,\dots,x_n)$ of $T$, assign $x_i$ the colour assigned to $v_i$. We obtain a proper $k$-colouring of $H^{(d)}$, and (2) holds. 

It is immediate that (3) $\Longrightarrow$ (2) $\Longrightarrow$ (1). 

It remains to prove that (1) $\Longrightarrow$ (4).
Suppose $H^{(d)}$ is $k$-colourable with defect $d-1$. For each vertex $x$ of $H^{(d)}$,  let $T_x$ be the subtree rooted at $x$ of the complete $d$-ary tree underlying $H^{(d)}$. Let $x_1$ be the root vertex. Say $x_1$ is blue. If each of the subtrees of $T$ rooted at the children of $x_1$ has a blue neighbour of $x_1$, then $x_1$ has monochromatic degree at least $d$. Thus there is a child $x_2$ of $x_1$ such that no neighbour of $x_1$ in $T_{x_2}$ is blue. Repeating this argument, we find a root-leaf path $x_1,\dots,x_n$ in $H^{(d)}$, such that for each vertex $x_i$, no neighbour of $x_i$ in $T_{x_{i+1}}$ is assigned the colour of $x_i$. Thus the copy of $H$ induced by $\{x_1,\dots,x_n\}$ is properly coloured with $k$ colours, and (4) holds. 
\end{proof}

\subsection{Triangle-Free Graphs}
\label{TriangleFree}

We now prove lower bounds for triangle-free graphs with given treedepth. By construction, $H^{(d+1)}$ has treedepth at most $|V(H)|$, and if $H$ is $K_p$-subgraph-free, then so is $H^{(d+1)}$. Thus \cref{GeneralisedStandardExample} implies:

\begin{cor}
\label{DefectiveTreedepthLowerBound}
The defective chromatic number of the class of $K_p$-free graphs of treedepth $k$ is at least the chromatic number of the class of $K_p$-free graphs on $k$ vertices.  
\end{cor}

\cref{DefectiveTreedepthLowerBound} applies for triangle-free graphs, but not for girth at least 5 (since $H^{(d)}$ contains a 4-cycle whenever $d\geq 2$ and $H$ contains a path $v_iv_kv_j$ with $i<j<k$). 

\citet{Kim95} showed that the Ramsey number $R(3,t)\in\Omega(t^2/\log t)$. That is, there exist triangle-free graphs $H$ on $n=\Omega(t^2/\log t)$ vertices with $\alpha(H)< t$. Thus 
\[\chi(H) \geq n/\alpha(H) > n/t \geq \Omega(t/\log t) \geq \Omega(\sqrt{n/\log n}).\] 
\cref{DefectiveTreedepthLowerBound} with $p=3$ and \cref{Parameters} thus imply:







\begin{cor}
\label{k/logk}
The defective chromatic number of the class of triangle-free graphs with treedepth $k$ (or pathwidth $k$ or treewidth $k$) is in $\Omega(\sqrt{k/\log k})$. 
\end{cor}

\begin{thm}
\label{KpFreeTredepthDefClusPropEqual}
Fix $p\geq 3$ and $k\geq 1$. 
Let $\GG$ be the class of $K_p$-free graphs of treedepth $k$.
Then the defective chromatic number, 
the clustered chromatic number, and
the chromatic number of $\GG$ are equal. 
\end{thm}

\begin{proof}
Obviously, the defective chromatic number of $\GG$ is at most the clustered chromatic number of $\GG$, which is at most the chromatic number of $\GG$. Let $c$ be the defective chromatic number of $\GG$. That is, for some $d$ every graph in $\GG$ is $c$-colourable with defect $d$. We now show that every graph in $\GG$ is properly $c$-colourable.

Consider a graph $H\in \GG$. So $H$ is a spanning subgraph of the closure of some rooted tree with vertex-height $k$. Define $\widehat{H}$ as follows. (This construction is an extension of the definition of $H^{(d)}$ above.) Initialise $\widehat{H}:=H$. We now make a sequence of changes to $\widehat{H}$, where at each step, we  maintain the property that $\widehat{H}$ is a spanning subgraph of the closure of some rooted tree with vertex-height $k$. In particular, for $i=2,3,\dots,k$ and for each vertex $w$ at depth $i$ in $\widehat{H}$, replace the current subgraph $Q_w$ of $\widehat{H}$ rooted at $w$ by $d+1$ copies of $Q_w$ rooted at $w$, and for every edge $xy$ where $x$ is an ancestor of $w$ and $y$ is a descendent of $w$ or equal to $w$, add an edge between $x$ and the copy of $y$ in each of the $d+1$ copies of $Q_w$. The final graph $\widehat{H}$ is in $\GG$. By assumption, $\widehat{H}$ is $c$-colourable with defect $d$. 

We now show that $\widehat{H}$ contains a properly $c$-coloured induced subgraph isomorphic to $H$. Let $T$ be the tree of vertex-height $k$ rooted at $r$ such that $H$ is a subgraph of the closure of $T$. For each vertex $v$ in $T$, we will choose exactly one vertex $v'$ in $\widehat{H}$ to represent $v$. Consider $w$ in non-decreasing order of $\dist_T(r,w)$. Let $r$ represent itself. Now assume $w\neq r$. Let $v$ be the parent of $w$ in $T$. So some vertex $v'$ in $\widehat{H}$ is already chosen to represent $v$. Say $v'$ is blue. In $\widehat{H}$ there are $d+1$ copies $w_1,\dots,w_{d+1}$ of $w$ with parent $v'$. 
If for each $i\in\{1,\dots,d+1\}$, the subgraph of $\widehat{H}$ rooted at $w_i$ has a blue neighbour of $v'$, then $v'$ has monochromatic degree at least $d+1$, which is a contradiction. Thus for some $i\in\{1,\dots,d+1\}$, the subgraph of $\widehat{H}$ rooted at $w_i$ has no blue neighbour of $v'$. Let  this vertex $w_i$  represent $w$. By construction, the chosen vertices induce a properly $c$-coloured subgraph of $\widehat{H}$ isomorphic to $H$. Hence every graph in $\GG$ is properly $c$-colourable. 
\end{proof}


We now use a variant of the approach of \citet{DK17} to improve the lower bound on the chromatic number in \cref{k/logk} to $(1-o(1))\sqrt{2k}$.

\begin{thm}
\label{TriangleFreeTreedepthProperLowerBound}
For any integer $k\geq 1$ there is a triangle-free graph with treedepth $\binom{k+1}{2}$ and chromatic number at least $k$. 
\end{thm}

\begin{proof}
Let $T$ be the tree rooted at a vertex $r$, and let $G$ be a subgraph of the closure of $T$ defined as follows. For each vertex $v$ of $T$ considered in non-decreasing order of $\dist_T(r,v)$, if $P_v$ is the $rv$-path in $T$, then for each independent set $I$ of $G[V(P_v)]$ introduce one child of $v$ with neighbourhood $I$. Consider any proper colouring of $G$. Let $\col(v)$ be the colour assigned to each vertex $v$. The following claim implies the theorem.

\textbf{Claim.} For each integer $k\geq 1$ there is a vertex $v$ in the first $\binom{k+1}{2}$ layers of $G$ and there is an independent set of $G[V(P_v)]$ assigned at least $k$ distinct colours.

We proceed by induction on $k$. The $k=1$ case is true with $v=r$. Assume the claim is true for some $k\geq 1$. Thus there is a vertex $w_0$ in the first $\binom{k+1}{2}$ layers of $G$ and there is an independent set $I$ of $G[V(P_{w_0})]$ assigned at least $k$ distinct colours.

By construction, there is a child $w_1$ of $w_0$ adjacent to no vertices in $I$. If no vertex in $I$ is assigned the same colour as $w_1$, then $I\cup\{w_1\}$ is an independent set of 
$G[V(P_{w_1})]$ assigned at least $k+1$ distinct colours, and $w_1$ is in the first $\binom{k+1}{2}+1\leq\binom{k+2}{2}$ layers of $G$. Now assume $\col(y_1)=\col(w_1)$ for some vertex $y_1\in I$.

We now construct a sequence of distinct vertices $w_1,\dots,w_{k+1}$ in $T$ and a sequence of distinct vertices $y_1,\dots,y_k$ in $I$, such that:
\begin{itemize}
\item for each $j\in\{1,\dots,k\}$, 
$\col(y_j)=\col(w_j)$
\item for each $j\in\{1,\dots,k+1\}$, $w_j$ is a child of $w_{j-1}$
\item for each $j \in \{1, \dots, k+1\}$, the vertex $w_j$ is adjacent to each vertex in  $\{y_1,\dots,y_{j-1}\}$ and is adjacent to no vertex in $(I\setminus\{y_1,\dots,y_{j-1}\})\cup\{w_1,\dots,w_{j-1}\}$. 
\end{itemize}
These properties hold for $w_1$ as shown above. 

Assume that $w_1,\dots,w_j$ and $y_1,\dots,y_{j-1}$ satisfy these properties for some $j\in\{1,\dots,k\}$. Thus $\col(w_j)\neq\col(y_i)=\col(w_i)$ for all $i\in\{1,\dots,j-1\}$. 
If no vertex in $I\setminus\{y_1,\dots,y_{j-1}\}$ is assigned the same colour as $w_j$, then $(I\setminus\{y_1,\dots,y_{j-1}\})\cup\{w_1,\dots,w_{j}\}$ is an independent set of $G[V(P_{w_j})]$ assigned $k+1$ distinct colours, where $w_j$ is in the first $\binom{k+1}{2}+j\leq \binom{k+2}{2}$ levels of $G$, and we are done. Thus we may assume that $\col(y_j)=\col(w_j)$ for some vertex $y_j$ in $I\setminus\{y_1,\dots,y_{j-1}\}$. By the definition of $G$, there is a child $w_{j + 1}$ of $w_j$ that 
is adjacent only to the independent set $\{y_1, y_2, \dots, y_j\}$. This allows us
to extend our sequence.

Repeating this step constructs the desired sequences. In particular,  
$\{w_1,\dots,w_{k+1}\}$ is an independent set of  $G[V(P_{w_{k+1}})]$ assigned at least $k+1$ distinct colours, and $w_{k+1}$ is in the first $\binom{k+1}{2}+k+1 = \binom{k+2}{2}$ layers of $G$, as claimed. 
\end{proof}

\cref{DK,KpFreeTredepthDefClusPropEqual,TriangleFreeTreedepthProperLowerBound} imply:

\begin{cor}
\label{TriangleFreeTreedepth}
For any integer $k\geq 1$ for the class of triangle-free graphs with treedepth $k$, the chromatic number, defective chromatic number and clustered chromatic number are equal, and are at least $(1-o(1))\sqrt{2k}$ and at most $\ceil{\frac{k+2}{2}}$.
\end{cor}

We have the following corollary of \cref{TriangleFreeTreedepthProperLowerBound} for graphs of given circumference, since $\cir(H)< 2^{\td(H)}$ for every graph $H$ (see \citep[Proposition~6.2]{Sparsity}). 

\begin{cor}
\label{CircumferenceLowerBound}
The chromatic number, defective chromatic number and clustered chromatic number of the class of triangle-free graphs with circumference $k$ is in $\Omega(\sqrt{\log k})$.
\end{cor}


Now consider triangle-free graphs with given treewidth.

\begin{thm}
\label{TriangleFreeTreewidth}
Fix $p\geq 3$ and $k\geq 1$. 
Let $\GG$ be the class of $K_p$-free graphs of treewidth $k$.
Then the defective chromatic number, 
the clustered chromatic number, and
the chromatic number of $\GG$ are equal. 
\end{thm}

\begin{proof}
Obviously, the defective chromatic number of $\GG$ is at most the clustered chromatic number of $\GG$, which is at most the chromatic number of $\GG$. Let $c$ be the defective chromatic number of $\GG$. That is, for some $d$ every graph in $\GG$ is $c$-colourable with defect $d$. We now show that every graph in $G\in\GG$ is properly $c$-colourable.

Consider a tree-decomposition $(B_x:x\in V(T))$ of $G$ with width at most $k$. Root $T$ at a node $r$. Normalise the tree-decomposition so 
that for each edge $xy\in E(T)$ with $x$ the parent of $y$, $|B_y\setminus B_x|=1$, and for each vertex $v$ of $G$ there is an edge $xy\in E(T)$ with $x$ the parent of $y$, such that $B_y\setminus B_x=\{v\}$. In this case, let $h(v):=y$ (the `home' of $v$). Let $\preceq$ be the partial order on $V(G)$, where $v\prec w$ if and only if $h(v)$ is an ancestor of $h(w)$ in $T$. The \defn{depth} of $v$ is $\dist_T(r,h(v))$. 

We now define a sequence of $K_p$-free graphs $G_1,G_2,\dots,G_m$, each with a rooted normalised tree-decomposition $(B_x:x\in V(T_i))$ of width at most $k$, where $G_1:=G$. In each graph $G_i$, each vertex will be `processed' or `unprocessed'. Initially, all vertices in $G_1$ are unprocessed. Suppose that $G_1,\dots,G_i$ are defined, and some vertex in $G_i$ is unprocessed. Let $v$ be an unprocessed vertex in $G_i$ at minimum depth. Let $y:= h(v)\in V(T_i)$. Let $U_1,\dots,U_m$ be the subtrees of $T_i$ rooted at the children of $y$. Let $V_j$ be the set of vertices $w\in V(G_i)$ where $h(w)\in V(U_j)$. Let $U:=\bigcup_{j=1}^m V_j$. Let $G_{i+1}$ be obtained from $G_i$ by making $d+1$ copies of each vertex in $U$. For each edge $pq$ of $G[U]$, add an edge to $G_{i+1}$ between the $\ell$-th copy of $p$ and the $\ell$-th copy of $q$, for each $\ell$. For each edge $wp\in E(G_i)$ where $w\in B_y$ and $p\in U$, add an edge between $w$ and the $\ell$-th copy of $p$, for each $\ell$. This includes the case $w=v$. 
Construct a normalised tree-decomposition of $G_{i+1}$ as follows. Let $T_{i+1}$ be obtained from $T_i$ by taking $d+1$ copies of each $U_j$, where the root of each copy is a child of $y$. For each vertex $p\in U$, if $p\in B_x$ then place the $\ell$-th copy of $p$ in the bag corresponding to the $\ell$-th copy of $x$ in $T_{i+1}$. We obtain a normalised tree-decomposition $(B'_x:x\in V(T_{i+1}))$ of $G_{i+1}$, such that for each vertex $w\in U$, if $h(w)=x$ then the home bags of the $d+1$ copies of $w$ are the $d+1$ copies of $x$. In $G_{i+1}$ say that $v$ is processed. We now show that $G_{i+1}$ is $K_p$-free. Suppose that $X$ is a set of vertices in $G_{i+1}$ that induce $K_p$. By the Helly property, $X$ is a subset of some bag $B'_x$. By construction, $G_{i+1}[B'_x]$ is isomorphic to the subgraph of $G_i$ induced by some bag $B_x$, implying that $G_i$ contains $K_p$. This contradiction shows that $G_{i+1}$ is $K_p$-free. 

This completes the definition of $G_1,G_2,\dots$. Note that for each vertex $v$ of $G$ the depth of each copy of $v$ always equals the original depth of $v$. In $G_i$, let $d_i$ be the minimum depth of an unprocessed vertex, and let $n_i$ be the number of processed vertices in $G_i$ at depth $d_i$. By construction, $(d_{i+1},n_{i+1})$ is lexicographically greater than $(d_i,n_i)$. Thus the sequence $G_1,G_2,\dots$ is finite. Let $G_n$ be the final graph, in which every vertex is processed. By construction, $G_n$ has treewidth at most $k$ and is $K_p$-free. 

Consider a $c$-colouring of $G_n$ with defect $d$. 
We now show that $G_n$ contains a properly $c$-coloured induced subgraph isomorphic to $G$. Say $v$ is the unprocessed vertex in $G_i$ chosen by the above process. 
Use the notation above. Let $y:= h(v)\in V(T_i)$. Let $U_1,\dots,U_m$ be the subtrees of $T_i$ rooted at the children $z_1,\dots,z_m$ of $y$. Let $V_j$ be the set of vertices $w\in V(G_i)$ where $h(w)\in V(U_j)$. Say $v$ is blue. For each $j\in\{1,\dots,m\}$, if each of the $d+1$ copies of $V_j$ in $G_{i+1}$ has a blue neighbour of $v$, then $v$ has defect $d+1$, which is a contradiction. So some copy of $V_j$ in $G_{i+1}$ has no blue neighbour of $v$. Say that this copy of $V_j$ and the corresponding edge of $T_{i+1}$ is `chosen' by $v$. This chosen edge remains in $T_n$. By construction, the union of the bags in $(B_x:x\in V(T_n))$ at nodes that are endpoints of the chosen edges in $T_n$ contains a properly $c$-coloured induced subgraph isomorphic to $G$. Therefore every graph in $\GG$ is properly $c$-colourable. 
\end{proof}

\cref{DK,TriangleFreeTreewidth} imply:

\begin{cor}
\label{TriangleFreeTreewidthDefectiveClusteredProper}
Let $\GG$ be the class of triangle-free graphs with treewidth at most $k$. Then the  defective chromatic number of $\GG$, the clustered chromatic number of $\GG$, and the proper chromatic number of $\GG$ all equal $\ceil{\frac{k+3}{2}}$.
\end{cor}

\subsection{Feedback Vertex Number}

We now prove two lower bounds on the clustered chromatic number of graphs with given feedback vertex number and given girth.

\begin{prop}
\label{FVN1TriangleFreeNotDefectLowerBound}
For each integer $d\geq 0$, there is a triangle-free graph $G$ with $\fvn(G)=1$ that is not 2-colourable with defect $d$.
\end{prop}

\begin{proof}
As illustrated in \cref{FVN1}, let $T_0$ be the $(d+1)$-leaf star. Let $T$ be obtained from $T_0$ by adding, for each vertex $v$ in $T_0$, a set $L_v$ of $2d+1$ leaves adjacent to $v$. Let $G$ be obtained from $T$ by adding one vertex $\alpha$ adjacent to all the leaves of $T$. So $\fvn(G)=1$ and $G$ is triangle-free (since $N_G(\alpha)$ is an independent set). Suppose that $G$ is 2-colourable with defect $d$. Say $\alpha$ is blue, and the other colour is red. For each vertex $v$ in $T_0$, at most $d$ vertices in $L_v$ are blue (since they are all adjacent to $\alpha$). So at least $d+1$ vertices in $L_v$ are red, implying $v$ is blue. Thus every vertex in $T_0$ is blue, which is a contradiction since $T_0$ has a vertex of degree $d+1$. Hence $G$ is not 2-colourable with defect $d$. 
\end{proof}

\begin{figure}[!ht]
\centering
\includegraphics{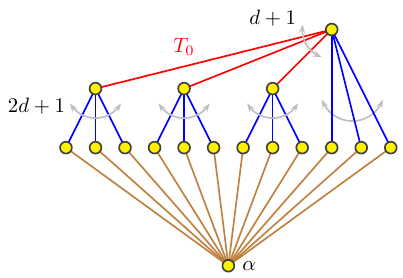}
\caption{Construction in \cref{FVN1TriangleFreeNotDefectLowerBound}.}
\label{FVN1}
\end{figure}

\begin{prop}
\label{FeedbackLowerBound}
For any $c\geq 1$ there is a graph $G$ with girth $6$ and $\fvn(G)\leq 7$, and $\tw(G)\leq \pw(G)\leq 8$, such that $G$ has no 2-colouring with clustering $c$.
\end{prop}

\begin{proof}
As illustrated in \cref{FVN7}, let $V$ be a set of 7 vertices. 
Let $n:=4c^2+4c+1$. 
For each set $S\in\binom{V}{4}$, 
add a path $P_S$ of length $n$ to $G$ (disjoint from $V$), where $P_S$ and $P_T$ are disjoint for distinct $S,T\in\binom{V}{4}$. 
Consider $S\in \binom{V}{4}$. 
For each vertex $x$ in $P_S$ add an edge to $G$ from $x$ to one vertex $v\in S$, so that for any distinct vertices $x,y\in V(P_S)$, if $xv,yv\in E(G)$ then $\dist_{P_S}(x,y)\geq 4$. This is possible since $|S|=4$. By construction, $\fvn(G)\leq 7$. In fact, $G-V$ is a forest of paths. Thus $\tw(G)\leq \pw(G)\leq 8$. 

\begin{figure}[!ht]
\centering
\includegraphics{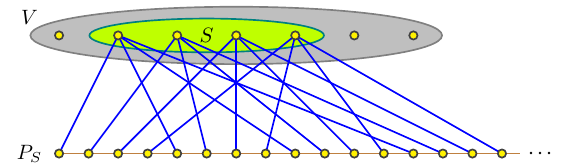}
\caption{Construction in \cref{FeedbackLowerBound}.}
\label{FVN7}
\end{figure}

Suppose for the sake of contradiction that $G$ has a cycle $C$ of length at most 5. Since $G-V$ is acyclic, $C\cap V\neq\emptyset$. Since $V$ is an independent set in $G$, $|C\cap V|\leq 2$. First suppose that $C\cap V=\{v\}$. Say $C=(v,a,b,c,d)$. Thus $(a,b,c,d)$ is a subpath of some $P_S$, and $\dist_{P_S}(a,d)\leq 3$, which contradicts the construction. Hence $|C\cap V|= 2$, say $C\cap V=\{v,w\}$. Since $V$ is an independent set, $v$ and $w$ are not consecutive in $C$. Thus $v$ and $w$ have a common neighbour in $C$, which must be in some $P_S$, but no vertex in $P_S$ is adjacent to two vertices in $V$, which is the desired contradiction. Hence $G$ has girth at least 6. In fact, $G$ has girth 6.

Suppose that $G$ has a 2-colouring with clustering $c$. There is a monochromatic subset $S$ of $V$ of size 4. Say $S$ is white. 
For each $v\in S$ there are less than $c$ white vertices in $P_S$ adjacent to $v$. Each vertex in $P_S$ is adjacent to some vertex in $S$. Thus there are less than $4c$ white vertices in $P_S$. The black subgraph of $P_S$ has at most $4c$ components. 
Hence $P_S$ contains a black subpath on $(n-4c)/4c>c$ vertices. This contradiction shows that $G$ has no 2-colouring with clustering $c$. 
\end{proof}

Note the following concerning the graph $G$ in \cref{FeedbackLowerBound}. Since $G-V$ is a forest of paths, it follows that for any $v,w\in V$, $G-(V\setminus\{v\})$ is outerplanar and so $G-(V\setminus\{v,w\})$ is planar. So $G$ is $5$-apex, which shows that $(\gamma,k)$-apex graphs with girth at most 6 have clustered chromatic number at least 3 (for $\gamma \geq 0$ and $k\geq 5$).

\section{Upper Bounds}

\subsection{Defective Colouring via Strong Colouring Numbers}
\label{StrongColouringNumbers}

\citet{KY03} introduced the following definition. For a graph $G$, a total order $\preceq$ of $V(G)$, and  $r\in\NN$, a vertex $w\in V(G)$ is \defn{$r$-reachable} from a vertex $v\in V(G)$ if there is a path $v=w_0,w_1,\dots,w_{r'}=w$ of length $r'\in[0,r]$ such that $w\preceq v$ and $v\prec w_i$ for all $i\in[r'-1]$. For a graph $G$ and integer $r\in\NN$, the \defn{$r$-strong colouring number} $\col_r(G)$ is the minimum integer such that there is a total order~$\preceq$ of $V(G)$ such that at most $\col_r(G)$ vertices are $r$-reachable from each vertex of $G$. 


An attractive aspect of strong colouring numbers is that they interpolate between degeneracy and treewidth \citep{KPRS16}. Indeed, it follows from the definition that $\col_1(G)$ equals the degeneracy of $G$ plus 1, implying $\chi(G)\leq \col_1(G)$. At the other extreme, \citet{GKRSS18} showed that $\col_r(G)\leq \tw(G)+1$ for all $r\in\NN$, and indeed $$\lim_{r\to\infty}\col_r(G)=\tw(G)+1.$$

Strong colouring numbers provide upper bounds on several graph parameters of interest, including 
acyclic chromatic number \citep{KY03}, 
game chromatic number \citep{KT94,KY03}, Ramsey numbers \citep{CS93}, oriented chromatic number \citep{KSZ-JGT97}, arrangeability~\citep{CS93}, etc. 
Strong colouring numbers are important also because they characterise bounded expansion classes \citep{Zhu09}, they characterise nowhere dense classes \citep{GKRSS18}, and have several algorithmic applications such as the constant-factor approximation algorithm for domination number by \citet{Dvorak13}, and the almost linear-time model-checking algorithm of \citet{GKS17}. 

Van~den~Heuvel~et~al.~\citep{HOQRS17} proved the following linear bounds on $\col_r(G)$ (always for every $r\in\NN$): Every finite planar graph $G$ satisfies $\col_r(G) \leq 5r+1$. More generally,  every finite graph $G$ with Euler genus $\gamma$ satisfies $\col_r(G) \leq (4\gamma+5)r + 2\gamma+1$. Even more generally, for $t\geq 4$, every $K_t$-minor-free graph $G$ satisfies $\col_r(G) \leq \binom{t-1}{2} (2r+1)$. 

\begin{thm}
\label{DefectiveColNumber}
Every $K_{s,t}$-subgraph-free graph is $s$-colourable with defect at most 
$\col_2(G)+(t-1)\binom{\col_2(G)}{s-1}$.
\end{thm}

\begin{proof}
Let $\preceq$ be a total ordering of $V(G)$ witnessing $\col_2(G)$. Consider each $v\in V(G)$ in the order given by $\preceq$, and choose $\col(v)\in \{1,\dots,s\}$ distinct from the colour assigned to the $s-1$ leftmost neighbours of $v$.
Suppose that $v$ has monochromatic degree at least  $\col_2(G)+(t-1)\binom{\col_2(G)}{s-1}+1$. At most $\col_2(G)$ neighbours of $v$ are to the left of $v$. Thus there is a set $W$ of $(t-1)\binom{\col_2(G)}{s-1}+1$ neighbours $w$ of $v$ with $\col(w)=\col(v)$ and $v\prec w$ for each $w\in W$. For each $w\in W$, by the choice of $\col(w)$, there is a set $A_w$ of $s-1$ neighbours of $w$ to the left of $v$. Thus $A:=\bigcup_{w\in W}A_w$ is 2-reachable from $v$. Hence $|A|\leq\col_2(G)$. For each $(s-1)$-subset $A'$ of $A$ there are at most $t-1$ vertices $w\in W$ with $A_w=A'$, as otherwise with $v$ there would be a $K_{s,t}$-subgraph in $G$. Hence $|W|\leq (t-1)\binom{\col_2(G)}{s-1}$, which is the desired contradiction. Thus each vertex $v$ has monochromatic degree at most $\col_2(G)+(t-1)\binom{\col_2(G)}{s-1}$. Hence $G$ is $s$-colourable with defect at most
$\col_2(G)+(t-1)\binom{\col_2(G)}{s-1}$.
\end{proof}

\citet{Zhu09} proved that $\nabla(G)$ and $\col_2(G)$ are tied. So \cref{DefectiveColNumber} and \cref{OOW} are equivalent (up to the defect term). Note that   \cref{OOW,DefectiveColNumber} both hold in the choosability setting.

We now give an application of \cref{DefectiveColNumber} for graphs of given circumference. We need the following tweak to \cref{DefectiveColNumber}. In a coloured graph $G$ a vertex $v$ is \defn{properly coloured} if $\col(v)\neq\col(w)$ for each edge $vw\in E(G)$; that is, $v$ is the only vertex in the monochromatic component containing $v$.

\begin{lem}
\label{DefectiveColNumberTweaked}
For any integer $t\geq s\geq 2$, for any $K_{s,t}$-subgraph-free graph $G$ and vertex $v\in V(G)$, there is an $s$-colouring of $G$ with defect at most 
$\col_2(G)+1+(t-1)\binom{\col_2(G)+1}{s-1}$, such that $v$ is properly coloured. 
\end{lem}


\begin{proof}
    Let $\preceq$ be the total order of $V(G)$ witnessing $\col_2(G)$. Let $\preceq'$ be the total order of $V(G)$ obtained from $\preceq$ by putting $v$ first. For each vertex $x\in V(G)$, the only vertex that is $2$-reachable from $x$ in $\preceq'$ but not in $\preceq$ is $v$. Thus at most $\col_2(G)+1$ vertices are 2-reachable from $x$ in $\preceq'$. Apply the colouring procedure from the proof of \cref{DefectiveColNumber} to $\preceq'$. So $v$ will be the first vertex coloured. If $vx\in E(G)$ then $v$ will be one of the $s-1$ leftmost neighbours of $x$, so $x$ will be assigned a colour distinct from the colour of $v$. That is, $v$ is properly coloured. The proof for the bound on the defect is analogous to the proof of \cref{DefectiveColNumber}.    
\end{proof}

\begin{thm}
\label{CircumferenceGirth5}
For every graph $G$ with circumference $k$ and girth at least 5, for every vertex $v$ of $G$, there is a 2-colouring of $G$ with clustering $(4k^2)^{k^2}$ such that $v$ is properly coloured. 
\end{thm}

\begin{proof}
%
We proceed by induction on $|V(G)|$. We may assume that $G$ is connected. If $G$ is not 2-connected, then $G$ has a separation $(G_1,G_2)$ with $V(G_1\cap G_2)=\{w\}$ for some vertex $w$. Without loss of generality, $v\in V(G_1)$. By induction, there is a 2-colouring of $G_1$ with clustering $c$ such that $v$ is properly coloured. By induction, there is a 2-colouring of $G_2$ with clustering $c$ such that $w$ is properly coloured. Permute the colours in $G_2$ so that $w$ is assigned the same colour in $G_1$ and $G_2$. We obtain a 2-colouring of $G$ with clustering $c$ such that $v$ is properly coloured.  

Now assume that $G$ is 2-connected. So $G$ has finite girth.
As mentioned above, $\col_2(G)\leq \tw(G)+1\leq\cir(G)\leq k$.
By \cref{DefectiveColNumberTweaked} with $s=t=2$, $G$ is 2-colourable with defect 
$\col_2(G)+1+(t-1)\binom{\col_2(G)+1}{s-1} \leq
2k+2$ such that  $v$ is properly coloured. \citet{Dirac52a} showed that every path in $G$ has length less than $k^2$ (using 2-connectivity). (In fact, \citet{BJMMSS} showed that $G$ has treedepth at most $k$.)\ Hence each monochromatic component has less than $(4k^2)^{k^2}$ vertices. 
\end{proof}

\subsection{Feedback Vertex Number}
\label{FVN}

We now prove upper bounds on the clustered chromatic number of graphs with given feedback vertex number and given girth.

\begin{prop}
Every graph $G$ with $\fvn(G)\leq 1$ and girth at least 5 is 2-colourable with clustering $2$.
\end{prop}

\begin{proof}
Let $v$ be a vertex of $G$ such that $G-v$ is a forest. Suppose that some leaf $\ell$ of $G-v$ is not adjacent to $v$. So $\deg_G(\ell)\leq 1$. By induction, $G-\ell$ is 2-colourable with clustering $2$. Assign $\ell$ a colour different from its neighbour. Now $G$ is 2-coloured with clustering $2$. Now assume that every leaf of $G-v$ is adjacent to $v$. Fix a root vertex $r$ of $G-v$. Let $\ell$ be a vertex in the same component of $G-v$ as $r$ at maximum distance from $r$. So $\ell$ is a leaf in $G-v$. Let $p$ be the parent of $\ell$ in $G-v$. So $pv\not\in E(G)$ as otherwise $(p,v,\ell)$ would be a 3-cycle in $G$. If $p$ has another child $\ell'$ (in addition to $\ell$), then $\ell'$ is a leaf by the choice of $\ell$, implying that $(\ell,p,\ell',v)$ is a 4-cycle in $G$. Hence $\ell$ is the only child of $p$, and $p$ has degree at most 2 in $G$. By induction, $G-\ell-p$ is  2-colourable with clustering 2. Assign $\ell$ a colour distinct from the colour assigned to $v$, and assign $p$ a colour distinct from the colour of the neighbour of $p$ other than $\ell$. Every monochromatic component created is contained in $\{\ell,p\}$. Thus the clustering is at most 2. 
\end{proof}

\begin{thm}
\label{FVNgirth7}
Every graph $G$ with $\fvn(G)\leq k$ and girth at least 7 is 2-colourable with clustering $\binom{k}{2}+k+1$.
\end{thm}

\begin{proof}
Let $A$ be a set of at most $k$ vertices in $G$ such that  $G-A$ is acyclic. Colour each vertex in $A$ white. Let $N:=N_G(A)$. Colour each vertex in $N$ black. Thus each monochromatic component intersecting $A$ is contained in $A$. Let $E:= E(G[N])$. For each edge $vw\in E$ there is a pair $\{a,b\}\in \binom{A}{2}$ such that $va,wb\in E(G)$; charge $vw$ to $\{a,b\}$. If distinct edges $v_1w_1,v_2w_2\in E$ are charged to the same pair $\{a,b\}\in \binom{A}{2}$, then $G[\{v_1,w_1,v_2,w_2,a,b\}]$ contains a cycle, contradicting that $G$ has girth at least 7. Thus no two distinct edges in $E$ are charged to the same pair in $\binom{A}{2}$. Hence $|E|\leq\binom{|A|}{2}$.

By assumption, $G-A$ is a forest. Root each component of $G-A$. Consider each component $C$ of $G-A-N$. Let $r_C$ be the root of $C$. Colour each vertex of $C$ at even distance from $r_C$ in $C$ white. Colour each vertex of $C$ at odd distance from $r_C$ in $C$ black. In particular, $r_C$ is white. Each monochromatic component of $C$ is a single vertex. The parent of $r_C$ in $G-A$ (if it exists) is in $N$, and is coloured black, which is different from the colour of $r_C$. No vertex in $C$ is adjacent to a vertex in $A$. Thus the white vertices in $C$ are properly coloured. 

Let $B$ be a black component of $G$. By construction, $B$ is a subgraph of $G-A$. So $B$ is a tree. Suppose that $|V(B) \setminus N|\geq 2$. Let $x,y$ be distinct vertices in $V(B)\setminus N$ at minimum distance in $B$. Since $G-A-N$ is properly coloured, $x$ and $y$ are not adjacent. The $xy$-path in $B$ must use the parent of $x$ or the parent of $y$ (possibly both). Without loss of generality, the parent $p$ of $x$ is in $B$. By the choice of $x$ and $y$, $p\in N$. Thus $x$ is the root of a component of $G-A-N$, implying that $x$ is white. This contradiction shows that $|V(B) \setminus N|\leq 1$. 
Say $V(B)\setminus N=\{x\}$. If $\deg_B(x)\geq |A|+1$ then there are two neighbours of $x$ both adjacent to the same vertex in $A$, implying that $G$ has a 4-cycle. Thus $\deg_B(x)\leq |A|$. Hence there at most $|A|$ edges in $B-N$. As shown above, $B[N]$ has at most $|E|\leq \binom{|A|}{2}$ edges. Thus $B$ has at most $\binom{|A|}{2}+|A|$ edges. Since $B$ is a tree, $B$ has at most 
$\binom{|A|}{2}+|A|+1$ vertices. Thus this 2-colouring of $G$ has clustering at most $\binom{|A|}{2}+|A|+1 \leq\binom{k}{2}+k+1$. 
\end{proof}





\subsection{Apex Graphs}
\label{ApexGraphs}

Here we prove clustered colouring results for $(\gamma,k)$-apex graphs with given girth. 

\begin{thm}
Every $(\gamma,k)$-apex graph with girth at least 9 is 2-colourable with clustering at most some function $f(\gamma,k)$.
\end{thm}

\begin{proof}
Let $G$ be an $n$-vertex $(\gamma,k)$-apex graph, where $n\geq 
n_0:= 14( \tfrac{9}7 \gamma + (k+2)\tbinom{k}{2})$. 
By \cref{gkApexEdges} below, $|E(G)| \leq \tfrac{25}{14}n + \tfrac97 \gamma + (k+2)\tbinom{k}{2} \leq \tfrac{13}{7}n$. By a result of \citet{AST90}, $G$ has a balanced separator of size $cn^{1/2}$ for some $c=c(\gamma,k)$. The result now follows from \cref{SSS} with $\epsilon=\frac12$ and $\delta=\frac17$ and $k=2$.
\end{proof}

\begin{lem}
\label{gkApexEdges}
For all integers $\gamma,k\geq 0$ every $n$-vertex $(\gamma,k)$-apex graph with girth at least 9 has at most $\tfrac{25}{14}n + \tfrac97 \gamma + (k+2)\tbinom{k}{2}$ edges.
\end{lem}

\begin{proof}
Let $G$ be an $n$-vertex $(\gamma,k)$-apex graph with girth at least 9. Thus there exists $A\subseteq V(G)$ with $|A|\leq k$ such that $G-A$ has Euler genus at most $\gamma$. By Euler's formula, $G-A$ has at most $\frac97(n-|A|+\gamma-2)$ edges. Let $N$ be the set of vertices in $G-A$ adjacent to at least one vertex in $A$. Let $X$ be the set of vertices in $N$ adjacent to at least two vertices in $A$. Thus $|X|\leq\binom{k}{2}$, as otherwise there exists two vertices in $X$ adjacent to the same two vertices in $A$, implying that $G$ contains a 4-cycle. Let $Y$ be the set of vertices in $N\setminus X$ that are at distance at most 2 in $G-A$ from another vertex in $N$. For each $v\in Y$ there are vertices $a,b\in A$ and there are edges $va,wb\in E(G)$ for some vertex $w\in N$ where $\dist_{G-A}(v,w)\leq 2$. In this case, charge  $v\in Y$ to the set $\{a,b\}$. No two vertices in $Y$ are charged to the same set $\{a,b\}\in\binom{A}{2}$, as otherwise $G$ contains a cycle of length at most 8. Thus $|Y|\leq\binom{k}{2}$. Let $Z:=N\setminus(X\cup Y)$. Every pair of vertices in $Z$ are at distance at least 3 in $G-A$. For each vertex $v\in Z$, if $v$ is isolated in $G-A$ then $\deg_G(v)\leq 1$, and by induction, 
\begin{align*}
|E(G)| \leq |E(G-v)|+1\leq 
\tfrac{25}{14}(n-1) + \tfrac97 \gamma + (k+1)\tbinom{k}{2} + 1 \leq
\tfrac{25}{14}n + \tfrac97 \gamma + (k+1)\tbinom{k}{2}.
\end{align*}
Now assume that no vertex in $Z$ is isolated in $G-A$. 
Thus $|N_{G-A}[v]|\geq 2$ for each vertex $v\in Z$, and 
$N_{G-A}[v] \cap N_{G-A}[w]=\emptyset$ for all distinct $v,w\in Z$. Hence $|Z|\leq \frac12(n-|A|)$. Now, 
\begin{align*}
|E(G)|
& \leq |E(G-A)| +|E(G[A])| + k |X| + |Y| + |Z| \\
& \leq
\tfrac97(n-|A|+\gamma-2) + \tbinom{k}{2} + (k+1)\tbinom{k}{2} + \tfrac12(n-|A|)\\
& \leq \tfrac{25}{14}n + \tfrac97 \gamma + (k+2)\tbinom{k}{2},
\end{align*}
as desired. 
\end{proof}

The next result shows that \cref{SSS} cannot prove 2-colourability for 1-apex graphs with girth 5. 

\begin{prop}
For infinitely many values of $n$ there is an $n$-vertex $1$-apex graph with girth 5 and with $(2-o(1))n$ edges.
\end{prop}

\begin{proof}
Let $G$ be an $n$-vertex planar square grid graph with vertices labelled by their coordinates $(x,y)$. Let $G'$ be the graph obtained from $G$ by subdividing each edge $(x,y)(x,y+1)$ where $x+y$ is even once. Each face of $G'$ has size 5, and any two division vertices are at distance at least 3 in $G'$. Let $G''$ be obtained from $G'$ by adding a vertex $v$ adjacent to every division vertex in $G'$. 
By construction, $G''$ is a $1$-apex graph. Moreover, $G''$ has girth 5, since any two division vertices are at distance at least 3 in $G'$. Note that $G$ has $(2-o(1))n$ edges and $(1-o(1))n$ faces. Thus $G'$ has $n$ original vertices, $(\frac12-o(1))n$ division vertices, and $(\frac52-o(1))n$ edges. Hence $G''$ has $(\frac32-o(1))n$ vertices and $(3-o(1))n$ edges. The result follows. 
\end{proof}

The next result shows that \cref{SSS} cannot prove 2-colourability for $k$-apex graphs with girth 8. 

\begin{prop}
For infinitely many values of $n$ there is an $n$-vertex $23$-apex graph with girth 8 and with $(2-o(1))n$ edges.
\end{prop}

\begin{proof}
Let $G$ be an $n$-vertex planar square grid graph.
Let $G'$ be the graph obtained from $G$ by subdividing each edge once. 
Each division vertex of $G'$ is at distance at most 5 from at most 22 other division vertices. Thus each division vertex of $G'$ can be coloured by one element of $\{1,\dots,23\}$ so that any two vertices of the same colour are at distance at least 6 in $G'$. Let $G''$ be obtained from $G'$ by adding a set of vertices $A:=\{v_1,\dots,v_{23}\}$ where each $v_i$ is adjacent to every vertex in $G'$ coloured $i$. By construction, $G''$ is a $23$-apex graph. We claim it has girth 8. clearly, every cycle in $G$ has length at least 8. Consider a cycle $C$ that intersects $A$. If  $C\cap A=\{v_i\}$, then $|C|\geq 8$ since 
vertices of $G'$ coloured $i$ are at distance at least 6 in $G'$.
If $C\cap A=\{v_i,v_j\}$, then $|C|\geq 8$ since vertices of $G'$ coloured $i$ and $j$ are at distance at least 2 in $G'$. Similarly, if $|C\cap A|\geq 3$ then $|C| \geq 12$. Hence $G$ has girth 8. Note that $G$ has $(2-o(1))n$ edges. Thus $G'$ has $n$ original vertices, $(2-o(1))n$ division vertices, and $(4-o(1))n$ edges. Hence $G''$ has $(3-o(1))n$ vertices and $(6-o(1))n$ edges. The result follows. 
\end{proof}

\begin{thm}
Every $(\gamma,k)$-apex graph with girth at least 5 is 3-colourable with clustering at most some function $f(\gamma,k)$.
\end{thm}

\begin{proof}
Let $G$ be an $n$-vertex $(\gamma,k)$-apex graph, where $n\geq 
n_0:= 6( \tfrac53 \gamma + (k+1)\tbinom{k}{2})$. By \cref{gkApexEdgesGirth5} below, $|E(G)| \leq \tfrac{8}{3}n + \tfrac53 \gamma + (k+1)\tbinom{k}{2} \leq \tfrac{17}{6}n$. By a result of \citet{AST90}, $G$ has a balanced separator of size $cn^{1/2}$ for some $c=c(\gamma,k)$. The result now follows from \cref{SSS} with $\epsilon=\frac12$ and $\delta=\frac16$ and $k=3$.
\end{proof}

\begin{lem}
\label{gkApexEdgesGirth5}
Every $n$-vertex $(\gamma,k)$-apex graph $G$ with girth at least 5 has at most $\tfrac{8}{3}n + \tfrac53 \gamma + (k+1)\tbinom{k}{2}$ edges.
\end{lem}

\begin{proof}
By assumption, there exists $A\subseteq V(G)$ with $|A|\leq k$ such that $G-A$ has Euler genus at most $\gamma$. By Euler's formula, $G-A$ has at most $\frac53(n-|A|+\gamma-2)$ edges. Let $N$ be the set of vertices in $G-A$ adjacent to at least one vertex in $A$. Let $X$ be the set of vertices in $N$ adjacent to at least two vertices in $A$. Thus $|X|\leq\binom{k}{2}$, as otherwise there exists two vertices in $X$ adjacent to the same two vertices in $A$, implying that $G$ contains a 4-cycle. Now, 
\begin{align*}
    |E(G)|\leq 
|E(G-A)| + |E(G[A])| + k |X| + |N| 
& \leq
\tfrac53(n-|A|+\gamma -2) + \tbinom{k}{2} + k\tbinom{k}{2} + n\\
& \leq \tfrac{8}{3}n + \tfrac53 \gamma + (k+1)\tbinom{k}{2},
\end{align*}
as desired. 
\end{proof}

\section{Increasing Girth}
\label{BigGirth}

\citet{Thomassen83} first showed that graphs with minimum degree 3 and sufficiently large girth contains a minor with large minimum degree. In particular, every graph with minimum degree at least 3 and girth at least $4k-3$ contains a minor with minimum degree $k$. \citet{Mader98} improved the bounds substantially, by showing that every graph with minimum degree at least 3 and girth at least $8k+3$ contains a minor with minimum degree $2^k$. 
\citet{KO03} improved the constant $8$ to $4$ as follows:


\begin{lem}[\citep{KO03}]
\label{MinDegree3Girth}
For any integers $k\geq 1$ and $r\geq 3$ every graph $G$ of minimum degree $r$ and girth at least $4k+3$ contains a minor of minimum degree at least
 $\frac{1}{48}(r-1)^{k+1}$.
\end{lem}

We use \cref{MinDegree3Girth} and the following lemma independently due to \citet{Kostochka82,Kostochka84} and \citet{Thomason84,Thomason01} to show that graphs excluding a fixed minor and with large girth are properly 3-colourable.

\begin{lem}[\citep{Thomason84,Thomason01,Kostochka82,Kostochka84}]
\label{KtMinorExtremal}
Every $K_t$-minor-free graph has minimum degree at most $ct\sqrt{\log t}$ for some absolute constant $c$. 
\end{lem}

\begin{prop}
\label{KtMinorGirth3Colours}
Every $K_t$-minor-free graph $G$ with girth at least
 $(4+o(1))\log_2 t$ is 3-colourable.
\end{prop}

\begin{proof}
The precise girth bound we assume is $4k+3$ where 
$k:=\floor{\log_2 (48ct\sqrt{\log t})}$
and $c$ is from \cref{KtMinorExtremal}.
We claim that $G$ is 2-degenerate. Otherwise, $G$ has a subgraph with minimum degree at least 3. By \cref{MinDegree3Girth}, we have $\frac{1}{48}2^{k+1} \leq ct\sqrt{\log t}$, implying
 $k \leq \log_2 (48ct\sqrt{\log t})-1$,
 which contradicts the definition of $k$.
 Hence $G$ is 2-degenerate and properly 3-colourable. 
\end{proof}

Of course, three colours is best possible in \cref{KtMinorGirth3Colours} since odd cycles are $K_4$-minor-free with arbitrarily large girth. 

\citet{KO03} used \cref{MinDegree3Girth} and a variant of \cref{KtMinorExtremal} with explicit bounds on the constant $c$ to conclude that $K_t$-minor-free graphs with girth at least $19$ are properly $(t-1)$-colourable, thus satisfying Hadwiger's Conjecture. In fact, the method of \citet{KO03} gives asymptotically stronger results for graphs of girth $g\geq 5$. To see this, we use the following lemma of \citet{KO03}, which is applicable for girth at least 5 whereas \cref{MinDegree3Girth} assumes girth at least 7.




\begin{lem}[\citep{KO03}]
\label{KO} 
For any integers $k\geq 1$ and $r\geq \max\{5k,2\,000\,000\}$, every graph $G$ of minimum degree at least $4r$ and girth at least $4k+1$ contains a minor of minimum degree at least
$\frac{1}{288}r^{(2k+1)/2}$.
\end{lem}

\begin{prop}
\label{KtMinorFreeGirthColour}
For any integer $k\geq 1$, every  $K_t$-minor-free graph $G$ with girth at least $4k+1$ is properly
$O( (t\sqrt{\log t})^{2/(2k+1)})$-colourable. 
\end{prop}

\begin{proof}
We claim that $G$ is $4r$-degenerate where
$$r:=\max\{\floor{(288ct\sqrt{\log t})^{2/(2k+1)}}+1,5k,2\,000\,000\}$$ 
and $c$ is from \cref{KtMinorExtremal}. Suppose that a subgraph of $G$ has minimum degree at least $4r$. By \cref{KtMinorExtremal,KO}, we have
$\frac{1}{288}r^{(2k+1)/2} \leq ct\sqrt{\log t}$, implying
$r \leq (288ct\sqrt{\log t})^{2/(2k+1)}$,
which contradicts the definition of $r$. 
Hence $G$ is $4r$-degenerate and
$\chi(G)\leq 4r+1
\in O( (t\sqrt{\log t})^{2/(2k+1)})$. 
\end{proof}

For example, \cref{KtMinorFreeGirthColour}  shows $\chi(G)\leq O( (t\sqrt{\log t})^{2/3})$ for girth $5$, and $\chi(G)\leq O( (t\sqrt{\log t})^{2/5})$ for girth $9$. 

For clustered colouring, we show that two colours suffice if we allow the smallest possible amount of clustering.

\begin{thm}
\label{GirthLogt}
Every $K_t$-minor-free graph with girth at least $(16+o(1))\log_2t$ is 2-colourable with clustering  2.
\end{thm}

\begin{proof}
The precise girth lower bound we use is $16\log_2(ct\sqrt{\log t})+11$, where $c$ is from \cref{KtMinorExtremal}. We proceed by induction on $|V(G)|$. First suppose that $\deg_G(v)\leq 1$ for some vertex $v$. By induction,
$G-v$ is 2-colourable with clustering at most 2. Colour $v$ by a colour not used by its neighbour. So 
$G$ is 2-coloured with clustering at most 2. Now assume that $G$ has minimum degree at least 2. Suppose that $G$ has an edge $vw$ with $\deg_G(v)=\deg_G(w)=2$. By induction, $G-v-w$ is 2-colourable with clustering at most 2. Colour $v$ by a colour not used by its neighbour different from $w$. Colour $w$ by a colour not used by its neighbour different from $v$. So the monochromatic component containing $v$ or $w$ has at most two vertices (namely, $v$ and $w$). Now  $G$ is 2-coloured with clustering at most 2. Now assume that for each edge $vw$ of $G$, at least one of $v$ and $w$ has degree at least 3. Let $g$ be the integer such that $G$ has girth in $\{8g+3,\dots,8g+10\}$. By \cref{Mader} below, there is a minor $G'$ of $G$ with minimum degree at least $2^{g/2}$. By \cref{KtMinorExtremal}, we have 
$2^{g/2}\leq ct\sqrt{\log t}$.
Hence $G$ has girth at most 
$8g+10\leq 10+16\log_2( ct\sqrt{\log t})$, which contradicts the starting assumption.
\end{proof}

It remains to prove the following lemma, which is a variant of the above-mentioned result of \citet{Mader98}. 

\begin{lem}
\label{Mader}
Let $G$ be a graph with:
\begin{itemize}
    \item girth at least $8k+3$ for some integer $k\geq 1$,
    \item minimum degree at least 2, and
    \item no edge $vw$ with $\deg_G(v)=\deg_G(w)=2$.
\end{itemize}
Then $G$ has a minor with minimum degree at least $2^{k/2}$. \end{lem}

\begin{proof}
Let $X$ be a maximal set of vertices in $G$ at pairwise distance at least $2k+1$. For each $x\in X$, let $B_x$ be the ball of radius $k$ around $x$. So $B_x\cap B_y=\emptyset$ for all distinct $x,y\in X$. Since $G$ has girth greater than $2k+1$, the induced subgraph $G[B_x]$ is a tree. By assumption, $G$ has minimum degree at least 2, and no edge $vw$ in $G$ has $\deg_G(v)=\deg_G(w)=2$. Thus $G[B_x]$ has at least  $2^{k/2}$ leaves. Let $V_i:=\{v\in V(G):\dist_G(v,X)=i\}$. By the maximality of $X$, we have $V(G)=\bigcup_{i=0}^{2k} V_i$. For each $x\in X$, let $B'_x$ be a superset of $B_x$ obtained as follows. For $i=k+1,k+2,\dots,2k$ each vertex $v\in V_i$ has a neighbour in $V_{i-1}\cap B'_x$ for some $x\in X$; add $v$ to $B'_x$. Thus $B'_x$ is a subset of the ball of radius $2k$ around $x$. Since $G$ has girth greater than $4k+1$, the induced subgraph $G[B'_x]$ is a (connected) tree. By construction,  $\bigcup_{x\in X}B'_x=V(G)$. Moreover, the number of edges between $B'_x$ and $\bigcup_{y\neq x}B'_y$ is at least the number of leaves in $G[B_x]$ which is at least $2^{k/2}$. For all distinct $x,y\in X$ there is at most one edge between $B'_x$ and $B'_y$, since otherwise $G$ would contain a cycle of length at most $8k+2$. Let $G'$ be obtained from $G$ by contracting each set $B'_x$ to a single vertex. Thus $G'$ has minimum degree at least $2^{k/2}$.
\end{proof}

\section{Open Problems}

Closing the gaps in the lower and upper bounds in \cref{TableTW,TableMinor} are natural open problems that arise from this work. In particular:

(1) For the class of triangle-free graphs with treedepth $k$, \cref{KpFreeTredepthDefClusPropEqual} says that the defective chromatic number,  clustered chromatic number and chromatic number are equal. But what is this value?  The best known lower bound is $(1-o(1))\sqrt{2k}$ (see \cref{TriangleFreeTreedepthProperLowerBound}).
\cref{DK} provides an upper bound of $\ceil{\frac{k+1}{2}}$. A $o(k)$ upper bound can be obtained as follows. 
\citet{DK17} showed that the maximum chromatic number of triangle-free graphs with treewidth $k$ equals the maximum online chromatic number of triangle-free graphs with pathwidth $k$. An analogous approach shows that the maximum chromatic number of triangle-free graphs of treedepth $k$ equals the maximum online chromatic number of triangle-free graphs on $k$ vertices. This result holds with the triangle-free condition replaced by $K_p$-free for any fixed $p$. Results of \citet{LST89} imply\footnote{\citet{LST89} describe an online algorithm that partitions an $n$-vertex graph into $\ceil{n/\log_2\log_2n}$ independent sets and at most $4n/\log_2\log_2n$ `residual' sets, where each residual set is a subset of the neighbourhood of a vertex. So for triangle-free graphs, each residual set is independent. Hence the online chromatic number of $n$-vertex triangle-free graphs is in $O(n / \log \log n)$. See \citep{GKMZ14} for related work.} that the online chromatic number of triangle-free graphs on $k$ vertices is in  $O(k/\log \log k)$. Thus, triangle-free graphs of treedepth $k$ are $O(k/\log \log k)$-colourable. As far as we are aware, it is possible that triangle-free graphs with treedepth $k$ are properly $O(\sqrt{k})$-colourable. This would strengthen the fact that triangle-free $k$-vertex graphs are properly $O(\sqrt{k})$-colourable, which is easily proved as a consequence of Brooks' Theorem.


(2) Are triangle-free $(\gamma,k)$-apex graphs 3-colourable with bounded clustering? Note that $K_{k+2,n}$ is a triangle-free $k$-apex graph with $(k+2)n$ edges. So \cref{SSS} does not help here. 

(3) Are graphs with girth at least some constant $g$ (possibly 7) and with bounded treewidth (or bounded pathwidth) 2-colourable with bounded clustering? The answer is `yes' for  graphs $G$ with treewidth~2 and girth at least 5, since such graphs are planar, so $|E(G)| < \frac53 |V(G)|$ by Euler's Formula, and the claim follows from \cref{SSS}. Perhaps \cref{SSS} can be used to answer the treewidth~3 case. But for larger values of treewidth, it seems a more tailored approach is needed. Note that $g\geq 7$ by \cref{FeedbackLowerBound}.

(4) Are $K_t$-minor-free graphs with girth at least some constant $g$ (possibly 5) 3-colourable with bounded clustering? Note that $g\geq 5$ by \cref{TriangleFreeTreedepth}. In fact, two colours may suffice. Are $K_t$-minor-free graphs with girth at least some constant $g$ (possibly 7) 2-colourable with bounded clustering? Here $g\geq 7$ by \cref{FeedbackLowerBound}.

(5) What is the clustered chromatic number of the class of triangle-free $K_t$-minor-free graphs. Without the triangle-free assumption, the answer is $t-1$, as proved by \citet{DEMW23} (and announced earlier by \citet{DN17}). The best lower bound is $\ceil{\frac{t+1}{2}}$, which follows from \cref{TriangleFreeTreewidthDefectiveClusteredProper} since graphs of treewidth $k$ are $K_{k+2}$-minor-free. 

Note that for proper colourings of $K_t$-minor-free graphs, improved bounds are known when we additionally exclude a fixed triangle or fixed complete graph. In particular, \citet{DP25} proved a $O(t)$ upper bound for $K_t$-minor-free graphs excluding any fixed complete subgraph.

{\fontsize{10pt}{11pt}
\selectfont
\def\soft#1{\leavevmode\setbox0=\hbox{h}\dimen7=\ht0\advance \dimen7
  by-1ex\relax\if t#1\relax\rlap{\raise.6\dimen7
  \hbox{\kern.3ex\char'47}}#1\relax\else\if T#1\relax
  \rlap{\raise.5\dimen7\hbox{\kern1.3ex\char'47}}#1\relax \else\if
  d#1\relax\rlap{\raise.5\dimen7\hbox{\kern.9ex \char'47}}#1\relax\else\if
  D#1\relax\rlap{\raise.5\dimen7 \hbox{\kern1.4ex\char'47}}#1\relax\else\if
  l#1\relax \rlap{\raise.5\dimen7\hbox{\kern.4ex\char'47}}#1\relax \else\if
  L#1\relax\rlap{\raise.5\dimen7\hbox{\kern.7ex
  \char'47}}#1\relax\else\message{accent \string\soft \space #1 not
  defined!}#1\relax\fi\fi\fi\fi\fi\fi}

}
\end{document}